\documentclass[10pt,a4paper]{article}
\usepackage{amsfonts}
\usepackage{amssymb}
\usepackage{mathrsfs}
\usepackage{amsthm}
\usepackage{setspace}
\usepackage{color}
\usepackage{mathbbol}

\usepackage{fancyhdr}
\newtheorem{theo}{Theorem}
\newtheorem{prop}[theo]{Proposition}

\newtheorem{lemm}[theo]{Lemma}

\theoremstyle{definition}

\newtheorem{defi}[theo]{Definition}

\newenvironment{lproof}{\emph{Proof of Lemma.}}{ \qed \par}

\newcommand{\be}{\begin{eqnarray*}}
\newcommand{\ee}{\end{eqnarray*}}
\newcommand{\beqa}{\begin{eqnarray}}
\newcommand{\eeqa}{\end{eqnarray}}
\newcommand{\ba}{\begin{array}}
\newcommand{\ea}{\end{array}}

\newcommand{\mf}{\mathfrak}

\newcommand{\wh}{\widehat}
\newcommand{\mbb}{\mathbb}

\newcommand{\wt}{\widetilde}

\begin{document}

\title{Subalgebras of Lie algebras with non-degenerate restriction of the Killing form}
\author{Stuart Armstrong}
\date{2007}
\maketitle
\begin{abstract}
Let $\mf{g}$ be any finite-dimensional Lie algebra with Killling form $B$. Let $\mf{h}$ be a subalgebra of $\mf{g}$ on which the Killing form is non degenerate. Then $\mf{h}$ is reductive.
\end{abstract}

\begin{theo}
Let $\mf{g}$ be any finite-dimensional Lie algebra with Killling form $B$. Let $\mf{h}$ be a subalgebra of $\mf{g}$ on which the Killing form is non degenerate. Then $\mf{h}$ is reductive.
\end{theo}
Any Lie algebra $\mf{h}$ is the semi-direct product of a semi-simple algebra $\mf{s}$ with a solvable algebra $\mf{k}$. We shall prove the result through a series of interim results.

If $\mf{g}$ is not a complex algebra, we may include it as $\mf{g} \times 1$ into $\mf{g}_{\mbb{C}} = \mf{g} \otimes_{\mbb{R}} \mbb{C}$; the Killing form of $\mf{g}_{\mbb{C}}$ resticted to $\mf{g}$ is twice the Killing form of $\mf{g}$ -- in particular, $\mf{h}$ remains non-degenerate. So without loss of generality, we may assume that $\mf{g}$ is a complex algebra.

Furthermore:
\begin{lemm}
Let $Z(\mf{g})$ be the centre of $\mf{g}$, and $\pi$ the projection $\mf{g} \to \mf{g} / Z(\mf{g})$. Then $\pi$ is injective on $\mf{h}$, the Killing form of $\mf{g}$ descends naturally to the Killing form of $\mf{g} / Z(\mf{g})$, and thus $\pi(\mf{h})$ is a subalgebra of $\mf{g} / Z(\mf{g})$, isomorphic with $\mf{h}$, on which the Killing form is non-degenerate.
\end{lemm}
\begin{lproof}
Note that $Z(\mf{g})$ is trivially an ideal, so $\pi$ is an algebra homomorphism. Let $g \in \mf{g}$, and $z \in Z(\mf{g})$. Then since $ad_z = 0$, $Z(\mf{g})$ is orthogonal, via the Killing form $B$ on $\mf{g}$, to all of $\mf{g}$. Thus $B$ descends to a bilinear form on $\mf{g}/Z(\mf{g})$. Since $\mf{h}$ is non-degenerate under $B$, $\mf{h} \cap Z(\mf{g}) = 0$ and thus $\pi$ is injective on $\mf{h}$.

It suffices to prove that $B$ is equal to the Killing form on the image algebra. But this is evident, as $ad_{g} Z(\mf{g}) = 0$, so the action of $ad_g$ descendes non-trivially to an action on $\mf{g}/Z(\mf{g})$ -- and this descended action is equal to the action of $ad_{\pi(g)}$. Moreover, taking the trace of $ad_{g_1} ad_{g_2}$ on $\mf{g}$ is the same as taking the trace of the action of $ad_{\pi(g_1)} ad_{\pi(g_2)}$ on $\mf{g}/Z(\mf{g})$.
\end{lproof}
By itterating this procedure as often as needed, we may assume that $\mf{g}$ has no centre -- equivalently, that the embedding $\mf{g} \subset \mf{gl}(\mf{g})$ via the adjoint representation is faithfull.

Let $N_V$ be the trace form on $\mf{gl}(V,\mbb{C})$ (i.e.~$N_V(A,C) =$ trace $AC$, with $AC$ seen as an endomorphism of $V$). Let $\rho$ be the adjoint representation of $\mf{g}$. Then by definition of the Killing form,
\be
B(X,Y) = N_{\mf{g}}(\rho(X) , \rho(Y)).
\ee
We will thus embbed $\mf{g}$ into $\mf{gl}(\mf{g})$, and use the trace form $N$ of this algebra. If we choose a basis $\{e_i\}$ of $\mf{g}$ with dual basis $\{e_i^*\}$, then $N$ is given by
\be
N(e_i \otimes e_j^*, e_k \otimes e_l^*) = \delta_{il}\delta_{jk}.
\ee
Since $\mf{k}$ is solvable, we may use conjugation to express it in upper-triangular form. This does not affect $N$; so assume that the $\{e_i\}$ are chosen so that $\mf{k}$ is upper-triangular. We will later need to construct a new basis in which ideals of $\mf{k}$ have a particularly simple form. To do so, we will need several results.

\begin{defi}
The space $A_{\lambda}^k$, the $k$-th extended eigenspace for the matrix $A$ with eigenvalue $\lambda$, is defined inductively as the space such that $A_{\lambda}^0 = 0$ and for all $v_k \in A_{\lambda}^k$,
\be
A(v_k) = \lambda v_k + v_{k-1},
\ee
for $v_{k-1}$ a section of $A_{\lambda}^{k-1}$. The maximal extended eigenspace for $\lambda$ is defined as $A_{\lambda} = lim_{k \to \infty} A_{\lambda}^k$.
\end{defi}

\begin{lemm}
If $A$ and $B$ commute, then $B$ maps all $A_{\lambda}^k$ to itself.
\end{lemm}
\begin{lproof}
Prove this by induction. Assume that $B$ maps $A_{\lambda}^{k-1}$ to itself (which is definetly true for $k = 1$), and let $v_k$ be a section of $A_{\lambda}^k$. Then
\be
A(B(v_k)) &=& B(A(v_k)) \\
&=& \lambda B(v_k) + B(v_{k-1}).
\ee
Since $B(v_{k-1})$ must be a section of $A_{\lambda}^{k-1}$, then $B(v_k)$ is thus a section of $A_{\lambda}^{k}$.
\end{lproof}

\begin{lemm} \label{decom:res}
For an abelian algebra $\mf{d} \subset \mf{gl}(\mf{g})$, define $\wh{\mf{d}}$ as the subset of $\mf{d}$ consisting of matrixes with the maximal number of distinct eigen-values (this is well defined, as the number of distinct eigenvalues is an integer valued function, bounded above by $d$, the dimension of $\mf{g}$). Then
\begin{itemize}
\item $\wh{\mf{d}}$ is open dense in $\mf{d}$,
\item for a given $A \in \wh{\mf{d}}$, $\mf{g}$ decomposes as a sum of $V_j = A_{\lambda_j}$ for distinct eigenvalues $\lambda_j$,
\item the above decomposition does not depend on the choice of $A$ in $\wh{\mf{d}}$,
\item all elements of $\mf{d}$ have a unique eigenvalue on $V_j$.
\end{itemize}
\end{lemm}
\begin{lproof}
Let $A \in \mf{d}$ have $m$ distinct eigenvalues, and $C \in \mf{d}$. Define the subset $\tau(A,C)$ of $\mbb{R}$ as $\{(\lambda_i^A - \lambda_j^A) / (\lambda_k^C - \lambda_l^C)\}$ for all the distinct eingenvalues $\lambda_i^A$ of $A$ and distinct eigenvalues $\lambda_k^C$ of $C$. This $S'$ is finite. Let $T(A,C) = \mbb{R} - \tau(A,C)$; it is an open dense set of $\mbb{R}$. Then if $x \in T(A,C)$,
\be
A' = A + xC,
\ee
must map each $A_k^1$ to itself (by definition), and must act on $A_k^1$ with at least one eigenvalue of the type $\lambda_k^A + x\lambda_j^C$ (since $A$ and $C$ commute, $C$ preserves $A_k^1$ so must have at least one eigenvector in $A_k^1$). By our choice of $x$, that eigenvalue is distinct for different $k$'s. Thus $A'$ has at least as many eigenvalues as $A$ does.

This shows that $\wh{\mf{d}}$ is open and dense in $\mf{d}$.

Now fix a given $A$ in $\wt{\mf{d}}$, and corresponding extended eigenspaces $V_j$ with $\mf{g} = \sum_j V_j$. Since $\mf{d}$ is commutative, every $C \in \mf{d}$ must map each $V_j$ to itself.

Now imagine that $C \in \mf{d}$ has two distinct eigenvalues $\lambda_1^C$ and $\lambda_2^C$ on a given $V_j$. Now $A$ must map $C_1^1 \cap A_j$ and $C_2^1 \cap A_j$ to themselves, and thus has eigenvectors on both these spaces; the eigenvalues must be $\lambda_j^A$. Then choosing $x \in T(A,C)$, we can see that
\be
A' = A + xC,
\ee
has at least as many distinct eigenvalues as $A$ on $V_k$, $k \neq j$, and has two distinct eigenvalues on $V_j$. Thus it has more distinct eigenvalues than $A$, a contradiction. From this we deduce that all $C \in \mf{d}$ must have a single eigenvalue on $V_j$. This further demonstrates that the definition of $V_j$ does not depend on the choice of $A$ in $\wt{\mf{d}}$.
\end{lproof}

We now return to proving the main theorem. Recall that $\mf{h}$ is the semi-direct product of a semi-simple algebra $\mf{s}$ with a solvable algebra $\mf{k}$.
\begin{prop}
The metric $B$ is non-degenerate on $\mf{k}$.
\end{prop}

\begin{proof}
Proof by contradiction -- this result is the heart of the overall proof. We shall be working within $\mf{h}$; so, for instance $\mf{k}^{\perp}$ is to be understood as $\mf{k}^{\perp} \cap \mf{h}$.

Let $\mf{l} = \mf{k}^{\perp} \cap \mf{k}$. Since $\mf{k}$ is an ideal of $\mf{h}$, so is $\mf{k}^{\perp}$ and hence so are $\mf{l}$ and $\mf{l}^{\perp}$.

Note that $\mf{k} \subset \mf{l}^{\perp}$. Consequently, $\mf{l}^{\perp} / \mf{k}$ is an ideal of $\mf{s}$. Since $\mf{s}$ is semi-simple, we know what such ideals are like; let $\mf{s} = \oplus_j \mf{s}_j$ for simple $\mf{j}$, with
\be
\mf{l}^{\perp} / \mf{k} = \oplus_{j>m} \mf{s}_j,
\ee
for some $m$, and set $\mf{t} = \oplus_{j=1}^m \mf{s}_j$. Now fix a given Lie algebra embedding $\mf{t} \subset \mf{h}$, and note that $N$ gives a non-degenerate pairing between $\mf{t}$ and $\mf{l}$. Now decompose the subalgebra $\mf{t} \oplus \mf{l}$ in terms of irreducible representations of $\mf{t}$. On the $\mf{t}$ component, this is the adjoint representation by definition; via the pairing, $\mf{t}$ must act on $\mf{l}$ via the adjoint representation as well (as the adjoint rep. is self-dual). Thus
\be
\mf{l} = \oplus_{j=1}^m \mf{s}_j',
\ee
where $\mf{s}_j$ acts on $\mf{s}_j'$ via the adjoint representation, and $\mf{s}_j$ acts trivially on $\mf{s}_k'$ for $j\neq k$. We thus have a map $\phi$ mapping $\mf{s}_j$ to $\mf{s}_j'$, with the properties that for elements $a$ and $b$ of $\mf{s}_j$,
\be
[a, \phi(b)] = \phi[a,b] = [\phi(a),b].
\ee
This demonstrates that $\mf{l}$ must be abelian: for $\mf{s}_j$ must preserve the bracket on $\mf{s}_j'$, hence must preserve the ideal $[\mf{s}_j', \mf{s}_j']$. This implies that $[\mf{s}_j', \mf{s}_j'] = \mf{s}_j'$ or $[\mf{s}_j', \mf{s}_j'] = 0$; since $\mf{s}_j'$ is solveable, the second equality must hold. Thus we can start using Lemma \ref{decom:res}.

Note, however, that because $N$ degenerates on $\mf{s}_j'$, but is preserved by the action of $\mf{s}_j$, it must be expressed, for elements $a$ and $b$ of $\mf{s}_j$, as
\be
N(a,b)&=& \alpha_j B_j(a,b) \\
N(\phi(a),b) &=& \beta_j B_j(a,b) \\
N(\phi(a),\phi(b)) &=& 0,
\ee
where $B_j$ is the Killing form of $\mf{s}_j$ and $\alpha_j$, $\beta_j$ are constants with $\beta_j \neq 0$.

We now want to have elements $A$ of $\mf{s}_j'$ and $C$ of $\mf{s}_j$ such that
\begin{itemize}
\item $A \in \wt{\mf{s}_j'}$,
\item $[C,A] \in \wt{\mf{s}_j'}$,
\end{itemize}
Since $[\mf{s}_j,\mf{s}_j] = \mf{s}_j$ and $\mf{s}_j' = \phi(\mf{s}_j)$, any non-zero element of $\wt{\mf{s}_j'}$ must be expressible in the form $[C,A']$ for some $C \in \mf{s}_j$ and $A' \in \mf{s}_j'$. Since $\wt{\mf{s}_j'}$ is open dense in ${\mf{s}_j'}$, we can choose $A \in \wt{\mf{s}_j'}$ near $A'$ such that $[C,A] \in \wt{\mf{s}_j'}$.

We may view $\mf{s}_j'$ as a subalgebra of $\oplus_j \mf{gl}(V_j)$, the $V_j$ defined as in Lemma \ref{decom:res}. Since the different $\mf{gl}(V_j)$ commute, the projection of $\mf{s}_j'$ to $\mf{gl}(V_j)$ remains reductive. Hence we may choose a basis $\{e_i\}$ of $\mf{g}$ such that each element of $\mf{s}_j'$ is upper-triangular in each $V_j \otimes V_j^*$ block, and has no entries outside these blocks.

We will now need to distinguish two cases: where $V_1 \neq \mf{g}$ (hence the splitting $\mf{g} = \sum_j V_j$ of Lemma \ref{decom:res} is non-trivial) and $V_1 = \mf{g}$. We shall deal with the first case first; here $A$ must have more than one distinct eigenvalue.

\begin{lemm}
$C$ has no lower diagonal entries in the $V_k \otimes V_l^*$ spaces for $k \neq l$.
\end{lemm}
\begin{lproof}
Again, proof by contradiction. Fix any $k$ and $l$, and note that $[C,A]$ must send $V_k \otimes V_l^*$ to itself. Let $C_{ij}$ be a non-zero, lower diagonal entry of $C$ in $V_k \otimes V_l^*$. We can choose $C_{ij}$ so that it has the minimal $j-i$ of all such possibilities. Now the strictly upper-triangular components of $A$ must acting on $C_{ij}$ will result in entries with strictly higher $j-i$, and the diagonal entries of $A$ will send $C_{ij}$ to $\lambda_k^A - \lambda_l^A C_{ij}$. All entries in $V_k \otimes V_l^*$ with same $j-i$ as $C_{ij}$ will be treated the same way.

Thus, since $A \in \wt{\mf{s}_j'}$, $\lambda_k^A - \lambda_l^A \neq 0$ and thus $[C,A]$ is not upper-triangular -- an impossibility as $\mf{s}_j'$ is an ideal of $\mf{s}_j \oplus \mf{s}_j'$. The result then follows.
\end{lproof}
Thus we may see $C$ as
\be
C = \sum_j C_j + U,
\ee
where $C_j$ are endomorphisms of $V_j$ and $U$ is strictly upper triangular. Now
\be
[C,A] = \sum_j [C_j,A_j] + [U,A].
\ee
We are now ready to derive the contradiction. $[C,A]$ is upper-triangular and $[U,A]$ is stricly upper-triangular, so has no diagonal entries. Since $[C,A] \in \wt{\mf{s}_j'}$, it must have non-zero diagonal entries (or else it only has the single eigen-value zero) -- thus there must be a $j$ such that $[C_j, A_j]$ has non-zero diagonal entries. Thus $[C_j,A_j]$ must be upper-triangular, trace free, and with a non-zero entry on the diagonal. This means that it has more than one distinct entry in the diagonal, hence more than one distinct eigenvalue. This contradicts the result of Lemma \ref{decom:res} that $[C,A]$, as an element of $\mf{s}_j'$, must have a single eigenvalue on $V_j$. So we have a contradiction, as long as $V_1 \neq \mf{g}$.

Now assume $V_1 = \mf{g}$, with $d$ the dimension of $\mf{g}$. This means that every element of $A'$ of $\mf{s}_j'$ has a single eigenvalue, $\lambda_{A'}$. And since $0 = N(A',A') = d(\lambda_{A'})^2$, we must have all eigenvalues of elements of $\mf{s}_j'$ as zero.

Pick an element $a$ of $\mf{s}_j$ so that $B_j(a,a) \neq 0$ -- hence $N(a,\phi(a)) \neq 0$. However $[a,\phi(a)] = \phi[a,a] = 0$. Set $A = \phi(a)$ and $C = a$. We may put $A$ in Jordan normal form, and since $A$ is an element of $\mf{s}_j'$ it will have only zeroes down the diagonal. So $A$'s only potentialy non-zero entries are $A_{i,i+1}$.

Now consider the diagonal terms of $[C,A]$. They must be
\be
\sum_{i=1}^{d-1} A_{i,i+1}C_{i+1,i} (e_i \otimes e_i^* - e_{i+1} \otimes e_{i+1}^*).
\ee
Now $N(C,A) = \sum_{i=1}^{d-1} A_{i,i+1}C_{i+1,i} \neq 0$, so there exists a minimum $i$ such that $A_{i,i+1}C_{i+1,i} \neq 0$. Then $[C,A]_{ii} = A_{i,i+1}C_{i+1,i} \neq 0$, contradicting the fact that $[C,A] = 0$.

Since this whole mess resulted from the assumption that $\mf{l}$ was non-zero, we must have $\mf{l} = 0$ and hence $N$ (and $B$) is non-degenerate on $\mf{k}$.
\end{proof}

\begin{lemm}
$\mf{k}$ is abelian.
\end{lemm}
\begin{lproof}
Choose a basis $\{e_j\}$ of $\mf{g}$ such that $\mf{k}$ is upper-triangular. Since $N$ is non-degenerate on $\mf{k}$, no element of $\mf{k}$ can be strictly upper-triangular (as then it would be orthogonal to $\mf{k}$). But $[\mf{k},\mf{k}]$ is striclty upper triangular; hence
\be
[\mf{k},\mf{k}] = 0.
\ee
\end{lproof}

\begin{prop}
$\mf{h}$ is the direct product of $\mf{s}$ with $\mf{k}$.
\end{prop}
\begin{proof}
We have $\mf{k}^{\perp}$ as an ideal in $\mf{h}$. Since $\mf{k}$ is non-degenerate, $\mf{k}^{\perp} \cap \mf{k} = 0$ and thus $\mf{k}^{\perp} \cong \mf{s}$. Let $s$ be any element of $\mf{k}^{\perp}$, and $k_1, k_2$ elements of $\mf{k}$. Then
\be
B([s,k_1],k_2) = -B(s,[k_1,k_2]) = 0,
\ee
as $\mf{k}$ is abelian. Since $\mf{k}$ is an ideal, $[s,k_1] \in \mf{k}$; since $B$ is non-degenerate on it, $[s,k_1] = 0$.

Thus the action of $\mf{k}^{\perp}$ on $\mf{k}$ is trivial, and
\be
\mf{h} = \mf{s} \times \mf{k}.
\ee
And so $\mf{h}$ is reductive.
\end{proof}

\end{document}